\documentclass[10pt]{article}

\usepackage{amsmath,amssymb,amsthm}
\usepackage{dsfont}
\usepackage{xcolor}
\usepackage{amsfonts}
\usepackage{algorithm}
\usepackage[colorlinks,linkcolor={blue},
citecolor={blue},urlcolor={red},]{hyperref}
\usepackage{hyperref}
\allowdisplaybreaks

\newtheorem{Theorem}{Theorem}[section]
\newtheorem{Proposition}[Theorem]{Proposition}
\newtheorem {Cor}[Theorem]{Corollary}
\newtheorem {pro}[Theorem]{Proposition}

\newtheorem {Lemma}[Theorem]{Lemma}
\newtheorem {rem}[Theorem]{Remark}
\newtheorem {rems}[Theorem]{Remarks}
\newtheorem {com}[Theorem]{Comment}
\newtheorem {coms}[Theorem]{Comments}
\newtheorem {warning}[Theorem]{Warning}
\newtheorem {notation}[Theorem]{Notation}
\newtheorem {Definition}[Theorem]{Definition}

\newtheorem {exer}[Theorem]{Exercise}

\newcommand{\bnota}{\begin{notation} \rm } \newcommand{\enota}{\end{notation}}
\newcommand{\bw}{\begin{warning} \rm } \newcommand{\ew}{\end{warning}}
\newcommand{\bcom}{\begin{com} \rm } \newcommand{\ecom}{\end{com}}
\newcommand{\bcoms}{\begin{coms} \rm } \newcommand{\ecoms}{\end{coms}}
\newcommand {\bdefi}{\begin{Definition}}
\newcommand {\edefi}{\end{Definition}}
\newcommand {\bl}{\begin{Lemma}}
\newcommand {\el}{\end{Lemma}}
\newcommand {\bethe}{\begin{Theorem}}
\newcommand {\eethe}{\end{Theorem}}
\newcommand {\bp}{\begin{pro}}
\newcommand {\ep}{\end{pro}}
\newcommand {\bcor}{\begin{Cor}}
\newcommand {\ecor}{\end{Cor}}
 \newcommand {\brem }{\begin{rem} \rm }
\newcommand {\erem }{\end{rem}}
 \newcommand {\brems }{\begin{rems} \rm }
\newcommand {\erems }{\end{rems}}
\newcommand {\bexo}{\begin{exer} \rm }
\newcommand {\eexo}{\hfill $\lhd$ \end{exer}}

\renewcommand{\tilde}{\widetilde}

\let\ssection=\section
\renewcommand{\section}{\setcounter{equation}{0}\ssection}
\usepackage{graphicx}
\usepackage[title]{appendix}
\usepackage{array,xtab,ragged2e}

\renewcommand{\footnoterule}{%
  \kern 2 pt
  \hrule width \textwidth width 2in
  \kern 2pt
}

\newcommand{\be}{\begin{equation}}
\newcommand{\ee}{\end{equation}}
\newcommand{\bde}{\begin{displaymath}}
\newcommand{\ede}{\end{displaymath}}
\newcommand{\beq}{\begin{eqnarray*}}
\newcommand{\eeq}{\end{eqnarray*}}
\newcommand{\beqa}{\begin{eqnarray}}
\newcommand{\eeqa}{\end{eqnarray}}
\newcommand{\bel }{\left\{\begin{array}{ll}}
\newcommand{\eel}{\cr \end{array} \right.}

\title{{Defaultable Perpetual Russian Option \\ Under a Last Exit Time Model}}
\author{Zhuoshu Wu\thanks{Corresponding author. Email address: zhuoshuwu@hotmail.com.}  and Libo Li 
\\ School of Mathematics and Statistics
\\ University of New South Wales
\\ NSW 2052, Australia
}

\begin{document}

\maketitle

\begin{abstract}
In this paper, we study a perpetual Russian option subject to default risk in the Black--Scholes framework, where the default event is modelled by the last time the asset price attains its running maximum. Since the default time is not a stopping time with respect to the natural filtration, we employ the associated Az\'{e}ma supermartingale to reformulate the original problem as a standard optimal stopping problem with a modified payoff structure. By introducing a suitable ratio process between the running maximum and the asset price, the problem is further reduced to a one-dimensional optimal stopping problem for a reflected diffusion. Using free-boundary techniques, we derive an explicit representation for the value function and show that the optimal stopping strategy is characterised by a threshold-type rule. The optimal stopping boundary is further identified as the unique solution to a nonlinear equation.\\

\end{abstract}

\section{Introduction}
Perpetual optimal stopping problems play a central role in mathematical finance, with the classical Russian option being a prominent example, see \cite{SheppandShiryaev1993}. In its standard formulation, the payoff depends on the running maximum of the underlying asset, and the problem admits an explicit solution via a reduction to a one-dimensional free-boundary problem. In this paper, we consider a defaultable version of the perpetual Russian option in the Black--Scholes market. 

In the default risk framework, default occurs at a random time $\theta$. Let us denote by $I\{\theta>\tau\}$, the indicator function of the set $\{\theta>\tau\}$, which equals $1$ if default occurs after the exercise time $\tau$ and $0$ otherwise. For a defaultable contingent claim, the promised payment is made only if default occurs after the exercise time $\tau$. If default occurs before the exercise time, the payment is not made, and the defaultable claim is worthless after the default time (zero recovery), see \cite{JR2000}. There are essentially two approaches to model a default time: either the default time is a stopping time in the asset's filtration, or it is a stopping time in a larger filtration. A comparative study of these approaches is contained in \cite{CM1996}. The main difference between these two approaches is that in the first approach, the default time is ``announced" whereas in the second approach, it is not, see \cite{EJY2000}.

Our default time $\theta$ is modelled by the second approach. It is the last exit time of the type studied in \cite[Section 4]{EJY2000}, where the default time is the last exit time of a diffusion at a fixed level, in our setting the level is not fixed but is the running maximum of the asset price. In particular, $\theta$ is not a stopping time in the filtration of the asset, so default is not announced, and the information about $\theta$ is carried by the conditional probability in Proposition \ref{ASTHETA}. Defaultable American-style contingent claims with default time modelled as a last exit time have been studied in several settings. When the default time is the last exit time of a fixed level, the defaultable American put options in both infinite and finite horizons are studied in  \cite{LW2024, WL2022} and the perpetual cancellable standard put and call options in \cite{GLZ2020}. When the default time is modelled as the last exit time of running maximum or minimum, the pricing of perpetual American standard and lookback put and call options is considered in \cite{GL2022, GL20222}.

The choice of $\theta$ is motivated by the structure of the Russian option, whose payoff is driven by the asset's running maximum and which is designed to reduce the regret of not having stopped at the maximum \cite{SheppandShiryaev1993}. Defining the default time $\theta$ as the last time the asset price is at its running maximum captures the exact moment at which the underlying asset price attains its historical peak, after which the asset price never returns to that level. In a defaultable Russian option setting, this creates a natural trade-off: the holder seeks to delay exercise to capture higher maximum, yet waiting too long risks exercising after the permanent peak has already occurred, rendering the claim worthless. Furthermore, because the market participants cannot observe in real time whether a current running maximum is the final one, $\theta$ is an honest time rather than a stopping time in the asset's filtration, which generates a default event purely from the asset's path properties.

This modelling framework naturally leads to additional analytical difficulties compared with the classical setting. To address them, we employ the associated Az\'{e}ma supermartingale studied by Mansuy and Yor \cite{MansuyandYor} and Nikeghbali \cite{Nikeghbali} to reformulate the original problem into a standard optimal stopping problem with a modified, nonlinear gain function. By introducing a suitable state process given by the ratio of the asset price to its running maximum, we further reduce the problem to a one-dimensional setting involving a reflected diffusion. Although the infinitesimal generator retains the same structure as in the classical Russian option pricing problem \cite{Shiryaev1999}, the nonlinear gain function induced by default risk leads to a different free-boundary equation. We analyse the resulting problem using the classical free-boundary techniques, as developed in \cite{PeskirandShiryaev}, and derive an explicit representation for the value function and identify the optimal stopping rule in threshold form. The optimal stopping boundary is characterised as the unique solution to a nonlinear equation arising from the smooth-fit condition.

The remainder of the paper is organised as follows. In Section 2, we formulate the model and derive the associated one-dimensional optimal stopping problem. In Section 3, we solve the corresponding free-boundary problem, establish the optimal stopping strategy and draw useful economic conclusions from the numerical experiments.

\section{Model and Reformulation} \label{TIFRABC4}

We consider a filtered probability space $(\Omega, \mathcal{F}, (\mathcal{F}_t)_{t\geq0}, \mathbb{P})$ satisfying the usual conditions, where the filtration $(\mathcal{F}_t)_{t\geq0}$ is generated by a standard Brownian motion $W = (W_t)_{t\geq0}$. We consider a complete market consisting of only one stock and one bond whose dynamics are given by the following stochastic differential equations:
\begin{align}
dX_t &= rX_t dt + \sigma X_t dW_t, \qquad X_0=x, \label{TDOX}\\
dB_t &= rB_t dt, \qquad\qquad\qquad\,\,\,\,\,\, B_0 = 1,
\end{align}
where $r>0$ is the interest rate and $\sigma>0$ is the volatility coefficient. We denote $S=(S_t)_{t\geq0}$ as the running maximum of the stock price process, i.e.,
\begin{align*}
S_t = s\vee \max_{u\in[0, t]} X_u,\qquad\qquad S_0 = s.
\end{align*}
Without loss of generality, we may assume that $x=1$ and $s\geq1$, see \cite{PeskirandShiryaev}.

The main optimal stopping problem can be reformulated as follows by the measurability of stopping time $\tau$ and the tower property:
\begin{align}
V &=\sup_{\tau} \mathbb{E}_{0, x, s} \left[ e^{-(r+\lambda)\tau} \left( S_\tau - L X_\tau \right)^+ I\{\theta>\tau\}\right] \nonumber\\
&=\sup_{\tau} \mathbb{E}_{0, x, s} \left[  \mathbb{E}_{0, x, s} \left[ e^{-(r+\lambda)\tau} \left( S_\tau - L X_\tau \right)^+ I\{\theta>\tau\} |\mathcal{F}_\tau \right]\right] \nonumber\\
&=\sup_{\tau}  \mathbb{E}_{0, x, s} \left[ e^{-(r+\lambda)\tau} \left( S_\tau - L X_\tau \right)^+ \mathbb{P}(\theta>\tau\big{|} \mathcal{F}_\tau) \right],\label{TOP}
\end{align}
where the supremum is taken over all the stopping times $\tau\in[0, \infty)$, $\theta = \sup\{t\geq 0: X_t = S_t\}$ is the last exit time and $\lambda>0$ is the discounting rate. Furthermore, we let $L>0$ and denote $I$ as the indicator function. The American option with payoff function $ \left( S_\tau - L X_\tau \right)^+$ is known as the Russian option (the payoff function is identical as that in \cite[Page 768]{Shiryaev1999}; additionally, we incorporate the indicator function to model the presence of default risk or early termination provisions in over-the-counter markets, see \cite{HW1995}.

First of all, we investigate the property of Az\'{e}ma supermartingale associated with $\theta$. 

\begin{Proposition} \label{ASTHETA}
In our setting, the following formula holds for $r-\frac{\sigma^2}{2}<0$ and $t\geq 0$,
\begin{align}
\mathbb{P}\left(\theta>t\big{|}\mathcal{F}_t \right) = \left(\frac{S_t}{X_t}\right)^\alpha, \label{IDOAZSU}
\end{align}
where $\alpha=\frac{2r}{\sigma^2}-1$. Moreover, \eqref{IDOAZSU} holds true with $t$ replaced by any $(\mathcal{F}_t)_{t\ge0}$-stopping time $\tau$ on $\{\tau<\infty\}$.
\end{Proposition}

\begin{proof}
We first notice that $\{\theta>t\}=\{\max_{u \geq t} X_u \geq S_t\}$, see  \cite{NY2006}.
Hence,
\begin{align}
\mathbb{P}(\theta>t | \mathcal{F}_t)&=\mathbb{P} \left(\max_{u \geq t} X_u \geq S_t | \mathcal{F}_t\right)=\mathbb{P}\left(\frac{\max_{u \geq t} X_u}{X_t} \geq \frac{S_t}{X_t} \bigg{|} \mathcal{F}_t\right)\nonumber\\
&=\mathbb{P}\left( \max_{v\geq0} \left[\left( r-\frac{\sigma^2}{2} \right)v + \sigma W_v\right] \geq \log{\frac{S_t}{X_t}} \right)\nonumber\\
&=1- \mathbb{P}\left(  \max_{v\geq0}\left[\left( r-\frac{\sigma^2}{2} \right)v + \sigma W_v\right] \leq \log{\frac{S_t}{X_t}} \right).\label{ROIIP1}
\end{align}
where the third equality follows from the facts that (i) processes $X_t$ and $S_t$ are $\mathcal{F}_t$-measurable and (ii) for $u=t+v$ and $v\geq0$, $ (W_{t+v} - W_t)_{v\geq0}$ is a Brownian motion independent of $\mathcal{F}_t$ (as $W$ has independent increments), with the same law as $(W_v)_{v\geq0}$. Then, for $r-\frac{\sigma^2}{2}<0$, $\sigma>0$ and $\log{\frac{S_t}{X_t}}\geq 0$, by \cite[Corollary 1 on Page 760]{Shiryaev1999}, 
\begin{align}
\mathbb{P}(\theta>t | \mathcal{F}_t)=1-(1- e^{\alpha \log{\frac{S_t}{X_t}} })=\left( \frac{S_t}{X_t} \right)^\alpha.
\end{align}
Let $\tau$ be a stopping time. On $\{\tau<\infty\}$, processes $S_\tau$ and $X_\tau$ are $\mathcal{F}_\tau$-measurable, and by the strong Markov property $(W_{u} - W_{\tau})_{u\geq\tau}$ is once again a Brownian motion independent of $\mathcal{F}_\tau$. Repeating the above computation with $t$ replaced by $\tau$ gives 
\begin{align}
\mathbb{P}(\theta>\tau| \mathcal{F}_\tau)=\left( \frac{S_\tau}{X_\tau} \right)^\alpha.
\end{align}
on $\{\tau<\infty\}$.
\end{proof}

\begin{rem} \label{RCOVERRO}
Here one should note that by the general result from \cite[Corollary 1 on Page 760]{Shiryaev1999}, for $r-\frac{\sigma^2}{2}\geq0$, $\mathbb{P}\left(\theta>t\big{|}\mathcal{F}_t \right) = 1$ indicating that no default occurs, problem \eqref{TOP} is therefore the classic Russian option pricing problem. For the detailed solution of pricing Russian option, see \cite[Page 400] {PeskirandShiryaev} and \cite{Shiryaev1999}. Hereafter, we assume $r-\frac{\sigma^2}{2}<0$ such that $\alpha<0$.
\end{rem}

Therefore, \eqref{TOP} gets the form
\begin{align}
V=\sup_{\tau} \mathbb{E}_{0, x, s} \left[ e^{-(r+\lambda)\tau} \left( S_\tau - L X_\tau \right)^+ \left( \frac{S_\tau}{X_\tau} \right)^\alpha \right].\label{STOP}
\end{align}

With some additional effort, we can reduce the above two-dimensional problem \eqref{STOP} into a one-dimensional problem. To do so, we introduce the probability measure $\tilde{\mathbb{P}}$ which satisfies $\left.\frac{d\tilde{\mathbb{P}}}{d\mathbb{P}}\right|_{\mathcal{F}_t}=e^{\sigma W_t - \frac{\sigma^2}{2}t}$, and the process $Y=(Y_t)_{t\geq0}=\left( \frac{S_t}{X_t} \right)_{t\geq0}$.

The strong solution of SDE \eqref{TDOX}, together with Girsanov's theorem, yields
\begin{align}
X_t=x e^{\left( r-\frac{\sigma^2}{2} \right)t +\sigma W_t } = x e^{\left( r+\frac{\sigma^2}{2} \right)t +\sigma \tilde{W}_t }, \label{TSSOGBMAC4}
\end{align}
for $t\geq0$, where $W$ and $\tilde{W}:=W_t - \sigma t$ are standard Brownian motions under measure $\mathbb{P}$ and $\tilde{\mathbb{P}}$ respectively.

\begin{Cor} \label{POPYC4}
Under probability measure $\tilde{\mathbb{P}}$, the process $Y = (Y_t)_{t\geq0}$ is a Markov process on the phase space $[1, \infty)$ with instantaneous reflection at the point $\{1\}$ and it satisfies the following SDE:
\begin{align}
dY_t = -r Y_t dt -\sigma Y_t d\tilde{W}_t + d R_t, \qquad Y_0=y=\frac{s}{x}, \label{TDOY}
\end{align}
where $d R_t=I\{ Y_t = 1 \} \frac{dS_t}{X_t}$. 
\end{Cor}
\begin{proof}
We follow the argument from \cite[Page 770]{Shiryaev1999}.

(i) First, we show that the process $Y = (Y_t)_{t\geq0}$ is Markov with respect to probability measure $\tilde{\mathbb{P}}$. Since $Y_t = \frac{S_t}{X_t}$, we have for $\delta>0$
\begin{align}
Y_{t+\delta} &= \max\bigg\{\frac{s}{X_{t+\delta}}, \frac{\max_{u\in[0, t+\delta]}X_u}{X_{t+\delta}} \bigg\} = \max\bigg\{\frac{s}{X_{t+\delta}}, \frac{\max_{u\in[0, t]}X_u}{X_{t+\delta}}, \frac{\max_{u\in(t, t+\delta]}X_u}{X_{t+\delta}} \bigg\} \nonumber\\
& = \max\bigg\{\frac{S_t/X_t}{X_{t+\delta}/X_t}, \frac{\max_{u\in(t, t+\delta]}X_u/X_t}{X_{t+\delta}/X_t} \bigg\} = \max\bigg\{\frac{Y_t}{X_{t+\delta}/X_t}, \frac{\max_{u\in(t, t+\delta]}X_u/X_t}{X_{t+\delta}/X_t} \bigg\}. \label{MCOY1}
\end{align}
Note that $\tilde{W}$ is a Brownian motion under measure $\tilde{\mathbb{P}}$ and by \eqref{TSSOGBMAC4}, we have for all $t<u\leq t+\delta$,
\begin{align}
\frac{X_u}{X_t} = e^{\left(r+\frac{\sigma^2}{2}\right)(u-t) + \sigma(\tilde{W}_u - \tilde{W}_t)},
\end{align}
which, together with \eqref{MCOY1}, implies that 
\begin{align}
\text{Law}(Y_{t+\delta} | \mathcal{F}_t; \tilde{P}) = \text{Law}(Y_{t+\delta} | Y_t; \tilde{P}),
\end{align}
and the desired Markov property follows.

(ii) Recall that $S_t = s\vee \max_{u\in[0, t]} X_u$, which is a nondecreasing process of locally bounded variation. Then, in view of \eqref{TSSOGBMAC4}, by exploiting It\^o formula, 
\begin{align*}
dY_t & =S_t d \left( \frac{1}{X_t}\right) + \left( \frac{1}{X_t}\right) dS_t = - \frac{S_t}{X_t^2} (rX_t dt + \sigma X_t dW_t) + \frac{S_t}{X_t^3} \sigma^2 X_t^2 dt + dR_t \\
&= (\sigma^2-r)Y_tdt - \sigma Y_t dW_t  + dR_t = -r Y_t dt - \sigma Y_t d\tilde{W}_t + dR_t,
\end{align*}
where $dR_t =\left( \frac{1}{X_t}\right) dS_t$; note that $dS_t(\omega) = 0$ on the set $\{ (\omega, u): Y_u(\omega) > 1\}$. We therefore write $R_t = \int_0^t I\{Y_u = 1\}\frac{dS_u}{X_u}$ to stress that the process changes value only when process $(Y_t)_{t\geq0}$ arrives at the boundary point $\{1\}$.

(iii) Finally, we show that, for any $t>0$, 
\begin{align}
\int_0^t I\{Y_u = 1\} du = 0, \,\,\tilde{\mathbb{P}}-a.s.
\end{align}
By Fubini's theorem,
\begin{align}
\tilde{\mathbb{E}} \int_{0}^\infty I\{ Y_u = 1 \} du = \int_0^\infty \tilde{\mathbb{E}} \left[ I\{ Y_u = 1 \}\right] du=  \int_0^\infty \tilde{\mathbb{P}} \left( Y_u = 1 \right) du = 0,
\end{align}
where the last equality is due to the probability distribution of the pair $(\max_{s\in[0, u]} \tilde{W}_s, \tilde{W}_u)$ has a density, see \cite[Page 772]{Shiryaev1999}. The process $Y$ spends ($\tilde{\mathbb{P}}$-a.s) zero time at point \{1\} and hence this point is an instantly reflecting boundary. 
\end{proof}

The infinitesimal generator of $(Y_t)_{t\geq0}$ is given by
\begin{align}
&\mathbb{L}_Y = -r y \frac{\partial}{\partial y} + \frac{1}{2}\sigma^2 y^2 \frac{\partial^2}{\partial y^2}, \qquad \text{in $(1, \infty)$}, \\
& \frac{\partial}{\partial y} = 0, \qquad \text{at $1+$}.
\end{align}
where the second condition means that the infinitesimal generator of $Y$ is acting on a space of $C^2$ functions $f$ defined on $[1, \infty)$ such that $f'(1+) = 0$, see \cite[Page 111]{SS1995} and \cite[Page 396]{PeskirandShiryaev}. Summarising our findings so far, we see that the problem \eqref{STOP}, by using the process $Y$ and change-of-measure, could be reduced to the following optimal stopping problem\footnote{For notational convenience, we simply write $\tilde{E}_{y}(A)$ instead of $\tilde{E}_{0, y}(A)$ hereafter.}
\begin{align}
V&=\sup_{\tau} \mathbb{E}_{0, x, s} \left[ e^{-(r+\lambda)\tau} X_\tau \left( \frac{S_\tau}{X_\tau} - L  \right)^+ \left( \frac{S_\tau}{X_\tau} \right)^\alpha \right]\nonumber\\
&=\sup_{\tau} \tilde{\mathbb{E}}_{0, y} \left[ e^{-\lambda\tau} \left( Y_\tau - L  \right)^+Y_\tau^\alpha \right].\label{TTOP}
\end{align}

\section{The Free-boundary Problem} \label{TIFTFPC4}

We are now ready to turn our attention to the \textit{free boundary problem} but first for the sake of brevity, we define the gain function $G(y)=(y-L)^+ y^\alpha$ and then we invoke the local time-space formula to obtain
\begin{align*}
e^{-\lambda t} G(Y_t^y)&=G(y)+\int_0^t e^{-\lambda u} \left( -\lambda G (Y_u^y) - r Y_u^y G_y (Y_u^y) + \frac{1}{2}\sigma^2 \left(Y_u^y\right) ^2 G_{yy} (Y_u^y) \right) I\{ Y_u^y > L \} du\\
&\qquad\,\,\,\,\,\,\,\,\,\,-\int_0^t e^{-\lambda u} \sigma Y_u^y G_y (Y_u^y) I\{ Y_u^y \neq L \} d\tilde{W}_u\\
&\qquad\,\,\,\,\,\,\,\,\,\,+\int_0^t e^{-\lambda u} G_y (1+) dR_u +\frac{1}{2}\int_0^t e^{-\lambda u} \left( G_y(L+) - G_y(L-)  \right) dl_u^L(Y),
\end{align*}
and 
\begin{align*}
G_y(L+) - G_y(L-) =
\begin{cases}
 L^{\alpha}, & L \geq 1,\\
 0, & L<1,
\end{cases} \qquad
G_y(1+) =
\begin{cases}
0, & L\geq 1, \\
\alpha(1-L) + 1, &L<1,
\end{cases}
\end{align*}
where $l_u^L(Y)$ is the local time of $Y$ at the level $L$ given by
\begin{align*}
l_u^L(Y)&=\tilde{\mathbb{P}}_y-\lim_{\epsilon\to0} \frac{1}{2\epsilon} \int_0^u I\{ \big{|}Y_r-L\big{|}<\epsilon \} d \langle Y, Y \rangle_r,
\end{align*}
from which, we note that when $L<1$, the point $L$ lies outside the state space of $Y$, hence the local time term vanishes, and we also observe that the further away $Y$ gets from $\max\{1, L \}$, the less likely the gain function will increase upon continuing, which suggests that there exists a point $b\in[\max\{1, L \}, \infty)$ such that the stopping time
\begin{align}
\tau_b=\inf\{ t\geq0: Y^y_t\geq b \} \label{OPST}
\end{align}
should be optimal in the problem \eqref{TTOP}, such fact will soon be confirmed. 

It is then natural to ask whether or not the stopping time $\tau_b$ is finite.

\begin{Cor} \label{OPSIF}
%The optimal stopping time $\tau_b$ is finite.
The stopping time $\tau_b$ is finite.
\end{Cor}

\begin{proof}
We essentially follow the proof from \cite{SheppandShiryaev1993}. To prove $\tilde{\mathbb{P}}_y(\tau_b<\infty)=1$ for $y\in[1, b)$, it suffices to show that 
\begin{align*}
\tilde{\mathbb{P}}_y \left( \max\limits_{t\geq0} Y_t \geq b \right)=1, \,\, \text{for $y=1$},
\end{align*}

To begin with, for $n\geq 1$, 
\begin{align*}
\tilde{\mathbb{P}}_y \left( \max\limits_{0 \leq t \leq n} \frac{S_t}{X_t} \geq b \right)& = \tilde{\mathbb{P}}_y \left( \max\limits_{0 \leq t \leq n} \frac{\max\limits_{0\leq u\leq t}  X_u}{X_t} \geq b \right) =  \tilde{\mathbb{P}}_y \left( \max\limits_{0 \leq u \leq t \leq n} \frac{X_u}{X_t} \geq b \right) \\
&=  \tilde{\mathbb{P}}_y \left( \max\limits_{0 \leq u \leq t \leq n} e^{\left( r+\frac{\sigma^2}{2}  \right) (u-t) + \sigma \left(\tilde{W}_u - \tilde{W}_t \right)} \geq b \right) \\
&\geq \tilde{\mathbb{P}}_y \left(  \max \bigg\{ \sigma \left(\tilde{W}_1 - \tilde{W}_0 \right), \dots, \sigma \left(\tilde{W}_{n} - \tilde{W}_{n-1} \right)  \bigg\}  \geq \log{b} +  r+\frac{\sigma^2}{2}   \right).
\end{align*}

Then, let $C= \frac{\log{b} +  r+\frac{\sigma^2}{2}}{\sigma} $, $X_n = \max \bigg\{  \tilde{W}_1 - \tilde{W}_0 , \dots, \tilde{W}_{n} - \tilde{W}_{n-1} \bigg\}$ and set event 
\[A_k=\{ \tilde{W}_{k+1} - \tilde{W}_{k} \geq C \}, \qquad \text{for $k\geq0$}.\]

Note that the events $\{A_k, \,\, 0 \leq k \leq n\}$ are independent and the crucial observation is that 
\begin{align*}
 \bigg\{ \lim_{n\to\infty} X_n \geq C    \bigg\} \Leftrightarrow \bigg\{ A_k \,\, \text{infinitely often} \bigg\}.
\end{align*}

Since $0 < \tilde{\mathbb{P}}_y(A_k) < 1$, it follows that $\sum_{k=0}^{\infty} \tilde{\mathbb{P}}_y (A_k) = + \infty$. An appeal to the second Borel-Cantelli lemma asserts that $\tilde{P}_y \left(A_k  \,\, \text{infinitely often}\right)=1$, and consequently that $\tilde{\mathbb{P}}_y \left(\max\limits_{t\geq0} Y_t \geq b\right) \geq \tilde{\mathbb{P}}_y \left(A_k  \,\, \text{infinitely often}\right) = 1$ as desired.
\end{proof}

By reviewing \cite[Page 40, Theorem 2.7]{PeskirandShiryaev}, we know that the next task is to find the smallest superharmonic function $\hat{V}$ that dominates $G$ and the unknown point $b$ by solving the corresponding free boundary problem:
\begin{align}
&\mathbb{L}_Y \hat{V} = \lambda \hat{V} \qquad\,\,\,\,\,\,\,\, \text{for $y\in(1, b)$}, \label{fbp1}\\
&\hat{V}(y) = G(y) \qquad\,\,\,\, \text{for $y=b$}, \label{fbp2}\\
&\hat{V}_y(y) = G_y(y) \qquad  \text{for $y=b$}, \qquad (\textit{smooth fit}), \label{fbp3}\\
&\hat{V}_y(y)=0 \qquad \qquad \text{for $y=1$}, \qquad (\textit{normal reflection}),\label{fbp4}\\
&\hat{V}(y)>G(y) \qquad\,\,\,\, \text{for $y\in[1, b)$}, \label{fbp5} \\
&\hat{V}(y)=G(y) \qquad\,\,\,\, \text{for $y\in(b, \infty)$}. \label{fbp6}
\end{align}
We proceed to solve the free-boundary problem, after which, we prove that its solution coincides with the value function in \eqref{TTOP} and $b$ is unique. 

First, we apply the infinitesimal generator of $Y$ in \eqref{fbp1} and obtain the Cauchy-Euler equation as follows
\begin{align}
 &\frac{1}{2} \sigma^2 y^2 \frac{d^2 \hat{V}}{d y^2} - ry \frac{d \hat{V}}{d y}  - \lambda \hat{V}= 0, \label{CEE}
\end{align}
from which, we know the solution takes the form
\begin{align}
&\hat{V}(y)=y^p. \label{gs}
\end{align}
By inserting \eqref{gs} into \eqref{CEE}, we obtain the following quadratic equation
\begin{align}
\frac{1}{2} \sigma^2  p^2 - \left( r+\frac{\sigma^2}{2} \right)p - \lambda =0, \label{QDE}
\end{align}
whose roots are given by
\begin{align*}
p_i&=\frac{\left(r +\frac{\sigma^2}{2}\right) \pm \sqrt{\left(r +\frac{\sigma^2}{2}\right)^2 + 2\lambda \sigma^2} }{\sigma^2}, \qquad\text{$(i=1, 2)$},
\end{align*}
where $p_1>1$, $p_2<0$. The general solution of \eqref{fbp1} therefore equals
\begin{align}
\hat{V}(y)= C_1 y^{p_1} + C_2 y^{p_2}, \label{GGS}
\end{align}
where $C_1$ and $C_2$ are arbitrary constants. Applying conditions \eqref{fbp2} and \eqref{fbp4} to \eqref{GGS} gives us
\begin{align*}
C_1= - \frac{p_2(b^{\alpha+1} - L b^\alpha)}{p_1 b^{p_2}- p_2 b^{p_1}},\qquad C_2= \frac{p_1(b^{\alpha+1} - L b^\alpha)}{p_1 b^{p_2}- p_2 b^{p_1}},
\end{align*}
so that
\begin{align}
\hat{V}(y)=
\begin{cases}
 \frac{(b^{\alpha+1} - L b^\alpha)}{p_1 b^{p_2}- p_2 b^{p_1}} \left( p_1 y^{p_2} - p_2 y^{p_1} \right),& \text{ $y\in[1, b]$},\\
 (y-L) y^\alpha, & \text{ $y\in[b, \infty)$},
\end{cases} \label{guusolutions}
\end{align}
and by using \eqref{fbp3}, we then know that $b\in(\max\{1, L\}, \infty)$ satisfies the following transcendental equation
\begin{align}
(\alpha+1) \left( p_1 b^{p_2} - p_2 b^{p_1}  \right)  - \alpha L \left( p_1 b^{p_2-1} - p_2 b^{p_1-1}  \right)  - p_1 p_2 (b-L) \left( b^{p_2-1} - b^{p_1-1}  \right) = 0. \label{FBF}
\end{align}

Thus, we have arrived at the following theorem:
\begin{Theorem} \label{THM1}
Suppose $r-\frac{\sigma^2}{2}<0$. The value function $V$ from \eqref{TTOP} is given explicitly by 
\begin{align}
{V}(y)=
\begin{cases}
 \frac{(b^{\alpha+1} - L b^\alpha)}{p_1 b^{p_2}- p_2 b^{p_1}} \left( p_1 y^{p_2} - p_2 y^{p_1} \right),& \text{ $y\in[1, b]$},\\
 (y-L) y^\alpha, & \text{ $y\in[b, \infty)$},\label{solutions}
\end{cases}
\end{align}
that is, $V=\hat{V}$. The stopping time $\tau_b$ from \eqref{OPST} with $b\in(\max\{1, L\}, \infty)$ being the unique solution to \eqref{FBF} above is optimal for the problem \eqref{TTOP}.
\end{Theorem}

\begin{proof}

(i) To prove the first part of Theorem \ref{THM1} is to verify that $V=\hat{V}$. Notice that this is the same as showing that $\hat{V}$ is the smallest superharmonic function dominating $G$.

We begin by showing that $\hat{V}(y) \geq G(y)$ for all $y\in[1, \infty)$. Note from \eqref{guusolutions} that for $y\in[b, \infty)$, $\hat{V}(y) = G(y)$ and for $y\in[1, L)$, $G(y) = 0$ and $\hat{V}(y)\geq 0$ and hence, $\hat{V}(y)\geq G(y)$ for $y\in [1, L)\cup[b, \infty)$. Let $h(y)=\hat{V}(y)-G(y)$ such that by \eqref{fbp3},
\begin{align*}
h'(b)= \hat{V}'(b)-G'(b)=0.
\end{align*}
We then wish to show that the stationary point $b$ is the global minimum of $h$ in $[\max\{1, L\}, b]$ so that $h(y)\geq h(b) = 0$ and thereby, proving that $\hat{V}(y)\geq G(y)$ in $[\max\{1, L\}, b]$. For this, we take the second derivative of $h$ and obtain for $y\in[\max\{1, L\}, b]$,
\begin{align*}
h''(y)&= \frac{b^\alpha (b-L) p_1 p_2}{p_1 b^{p_2}-p_2 b^{p_1}} \left( (p_2-1)y^{p_2-2} - (p_1-1)y^{p_1-2} \right)\\
&\qquad \qquad - \frac{y^{\alpha-2}}{\sigma^4} \left( (2r - \sigma^2) \left(2ry + L(2\sigma^2-2r)\right)  \right)>0,
\end{align*}
where the strict inequality follows from the fact that $p_1>1$, $p_2<0$, $b>L$ and $r-\frac{\sigma^2}{2}<0$, indicating the convexity of function $h$ in $[\max\{1, L\}, b]$. Since $y\mapsto h'(y)$ is increasing in $[\max\{1, L\}, b]$ and $h'(b) = 0$, it must be that on this interval, $h'(y)\leq0$, which suggests that $y\mapsto h(y)$ is decreasing and therefore, $h(y)\geq h(b)$ in $[\max\{1, L\}, b]$. By the continuity of functions $\hat{V}$ and $G$, we conclude that $\hat{V}(y)\geq G(y)$ for all $y\in[1, \infty)$ 

Next we are ready to show that $V(y)\leq \hat{V}(y)$ for all $y\in[1, \infty)$. From \eqref{solutions}, it is fairly obvious that $e^{-\lambda t} \hat{V}(y)$ is $C^{1,2}$ on $\mathcal{C}$ and $\mathcal{\bar{D}}$, where 
\begin{align*}
\mathcal{C}&=\{ (t, y)\in[0, \infty)\times[1,\infty): y<b \},\\
\mathcal{\bar{D}}&=\{ (t, y)\in[0, \infty)\times[1,\infty): y>b \},
\end{align*}
and therefore, we exploit the local time-space formula to obtain
\begin{align}
&e^{-\lambda t}\hat{V}(Y_t)= \hat{V}(y) + \int_0^t e^{-\lambda u} \left( \mathbb{L}_Y \hat{V} -\lambda \hat{V}  \right) (Y_u) du \nonumber\\
&\qquad +  \int_0^t e^{-\lambda u}  \hat{V}_y(Y_u) dR_u -  \int_0^t e^{-\lambda u}  \sigma Y_u \hat{V}_y(Y_u) d\tilde{W}_u \nonumber\\
&=\hat{V}(y) + \int_0^t e^{-\lambda u} \left( \mathbb{L}_Y \hat{V} -\lambda \hat{V}  \right) (Y_u) du  -  \int_0^t e^{-\lambda u}  \sigma Y_u \hat{V}_y(Y_u) d\tilde{W}_u, \label{IFV}
\end{align}
where the first equality follows via the smooth-fit condition \eqref{fbp3} and the second equality holds as $\hat{V}$ satisfies the normal reflection condition \eqref{fbp4}.

Since $\hat{V}(y)=G(y)=(y-L) y^\alpha$ in $\mathcal{\bar{D}}$, we have 
\begin{align*}
(\mathbb{L}_Y G-\lambda G)(y) &= - (\lambda+r) y^{\alpha+1} + L (\lambda +2r - \sigma^2) y^\alpha \nonumber\\
&=  (\lambda + r) (L - y) y ^{\alpha} + (r - \sigma^2) L y^\alpha < 0,
\end{align*}
where we have used the fact that  $y>b>L$ to conclude that the first term is negative, and $r-\frac{\sigma^2}{2}<0$ for the second term. This, together with \eqref{fbp1}, shows that
\begin{align}
\mathbb{L}_Y \hat{V} - \lambda \hat{V} \leq 0, \label{LLV}
\end{align}
everywhere on $[1, \infty)$ but $b$. As $\tilde{\mathbb{P}}_y(Y_u=b)=0$, we thus have
\begin{align}
 e^{-\lambda t} G(Y_t) \leq e^{-\lambda t} \hat{V}(Y_t) \leq \hat{V}(y) + \tilde{M}_t, \label{SIEV}
\end{align}
where the first inequality follows from the first observation that $\hat{V}\geq G$, the second inequality holds via \eqref{IFV} and \eqref{LLV}, and moreover, $\tilde{M}_t =-  \int_0^t e^{-\lambda u}  \sigma Y_u \hat{V}_y(Y_u) d\tilde{W}_u$ is a continuous local martingale. Note that $0\leq \hat{V}'(y) \leq L^{\alpha}$ for $y\in[1, \infty)$, which follows from the fact that $\hat{V}(y) = G(y)$ (and $G'(y)\leq L^\alpha$ by direct computation) in $[b, \infty)$ and that $\hat{V}''(y)>0$ in $[1, b]$, that is, $y\mapsto \hat{V}'(y)$ is increasing and given that $\hat{V}'(b) = {G}'(b)$ and $\hat{V}'(1) = 0$, it's clear that $0\leq \hat{V}'(y) \leq \hat{V}'(b) = G'(b)$. Additionally, $\tilde{\mathbb{E}}_y[\sup_{u\leq t} Y_u^2] < \infty$ for every $t$, since
\begin{align*}
Y_u^2 = \max\bigg\{ye^{-\left(r+\frac{\sigma^2}{2}\right)u - \sigma \tilde{W}_u}, \sup_{v\leq u} e^{\left(r+\frac{\sigma^2}{2}\right)(v-u) + \sigma(\tilde{W}_v - \tilde{W}_u)} \bigg\}^2 \leq y^2 e^{2\sigma |\tilde{W}_u|} + e^{4\sigma \sup_{w\leq u}|\tilde{W}_w|},
\end{align*}
where the second inequality is due to $r+\frac{\sigma^2}{2}>0$ and $v-u\leq 0$. Consequently, $\tilde{M}$ is a square integrable martingale on $[0, t]$.

Remember that a stopped martingale does not always remain a martingale, but for bounded stopping times, it always preserves its martingale property, see \cite[Page 495]{Gut2006}.  Therefore, let $\tau_n=\tau\wedge n$ be a bounded stopping time, for $n\geq 1$ so that $\tilde{\mathbb{E}}_y \left[ \tilde{M}_{\tau_n} \right]=0$ by the optional sampling theorem. Then, taking the expectation under measure $\tilde{\mathbb{P}}_y$ gives us
\begin{align*}
\tilde{\mathbb{E}}_y\left[ e^{-\lambda \tau_n} G(Y_{\tau_n}) \right] \leq \hat{V}(y). 
\end{align*}
Now, let $n\to\infty$, so that $\tau_n\to\tau$. We invoke Fatou's lemma to obtain
\begin{align*}
\tilde{\mathbb{E}}_y \left[ e^{-\lambda \tau} G(Y_\tau) \right] \leq \liminf_{n\to\infty} \tilde{\mathbb{E}}_y\left[ e^{-\lambda \tau_n} G(Y_{\tau_n}) \right] \leq \hat{V}(y), 
\end{align*}
after which, by taking the supremum over all stopping times $\tau$ of $Y$, together with \eqref{TTOP}, we prove that $V(y)\leq \hat{V}(y)$ for all $y\in[1, \infty)$.

To finish off, we have to show that $V(y)=\hat{V}(y)$ for all $y\in[1,\infty)$. Let $\tau_n = \tau_b \wedge n$ for $n\geq0$ and $\tau_b$ be the finite stopping time defined in \eqref{OPST}. Then, set $t=\tau_n$ in \eqref{IFV} so that
\begin{align*}
e^{-\lambda\tau_n} \hat{V}\left( Y_{\tau_n} \right) = \hat{V}(y) + \tilde{M}_{\tau_n},
\end{align*}
of which, taking the expectation under $\tilde{P}_y$, upon using the same argument as before yields $\tilde{\mathbb{E}}_y \left[ \tilde{M}_{\tau_n} \right] = 0$ and letting $n\to\infty$,
\begin{align*}
\tilde{\mathbb{E}}_y \left( e^{-\lambda\tau_b} \hat{V}\left( Y_{\tau_b} \right) \right)= \hat{V}(y).
\end{align*}
Furthermore, we conclude that, in view of the fact that $\hat{V}(Y_{\tau_b})= G(Y_{\tau_b})$ and \eqref{TTOP},
\begin{align*}
\hat{V}(y)=\tilde{\mathbb{E}}_y \left( e^{-\lambda\tau_b} G\left( Y_{\tau_b} \right) \right) \leq \sup_{\tau} \tilde{\mathbb{E}}_y \left( e^{-\lambda\tau} G\left( Y_{\tau} \right) \right)  = V(y),
\end{align*}
 for all $y\in[1,\infty)$, which, joining with the fact that $\hat{V}(y)\geq V(y)$, proves that the equality holds true. With $\tau_b$ being finite, \cite[Theorem 2.7, Page 41]{PeskirandShiryaev} then provides us with the positive answer for $\tau_b$ being optimal.

(ii) It thus remains to show the second part of the Theorem \ref{THM1}, that is $b$ is unique, i.e. the equation \eqref{FBF} has only one root. Let
\begin{align*}
g(y) &= (\alpha+1) \left( p_1 y^{p_2} - p_2 y^{p_1}  \right)  - \alpha L \left( p_1 y^{p_2-1} - p_2 y^{p_1-1}  \right) \nonumber\\
&\qquad \qquad \qquad - p_1 p_2 (y-L) \left( y^{p_2-1} - y^{p_1-1}  \right),
\end{align*}
for $y\in\left(\max\{1, L\}, \infty\right)$. Then, upon using  $p_1 p_2 = \frac{-2\lambda}{\sigma^2}$, we have
\begin{align*}
g'(y) & = \frac{2\lambda}{\sigma^2 y^2} (L-y) \left( (p_1-1) y^{p_1} + (1- p_2) y^{p_2} \right) \\
& \qquad \qquad + \frac{\alpha}{y^2} \left( \frac{2\lambda}{\sigma^2} (y- L)(y^{p_1} - y^{p_2}) + L(p_1 y^{p_2} - p_2 y^{p_1}) \right)<0,
\end{align*}
for $y\in\left(\max\{1, L\}, \infty\right)$, where the strict inequality holds via (i) $p_1 > 1$, $p_2<0$ and $y>L$ for the first term to be strictly negative and (ii) $\alpha<0$, $y^{p_1} - y^{p_2} > 0$ since $y>1$ and $p_1>p_2$ and $p_1 y^{p_2} - p_2 y^{p_1} > 0$ as $p_1>0>p_2$; which implies that the map $y\mapsto g(y)$ is decreasing. 

To apply the intermediate value theorem, we need to further estimate the value of the endpoints:
\begin{align*}
g(1) = \left( \frac{2r}{\sigma^2} + L \left( 1- \frac{2r}{\sigma^2} \right) \right) (p_1 - p_2)>0, \qquad
g(L) = p_1 L^{p_2} - p_2 L^{p_1} > 0,
\end{align*}
where the first strict inequality follows by $\frac{2r}{\sigma^2}-1<0$. 

From direct computation, it is clear that
\begin{align*}
g(y) &= \underbrace{ (\alpha + 1) p_1 y^{p_2} - \alpha L p_1 y^{p_2 - 1} - p_1 p_2 (y-L) y^{p_2 - 1}}_{\mathbf{I}}\\
&\qquad \underbrace{- (\alpha + 1) p_2 y^{p_1} + \alpha L p_2 y^{p_1 - 1} + p_1 p_2 (y-L) y^{p_1 - 1}}_{\mathbf{II}},
\end{align*}
Let $y\to\infty$, then as $p_2<0$, we have $\lim_{y\to\infty} \mathbf{I} = 0$ and upon using $p_1p_2 = -\frac{2\lambda}{\sigma^2}$, we have
\begin{align}
\lim_{y\to\infty} g(y) &= \lim_{y\to\infty} \mathbf{II}= \lim_{y\to\infty} \frac{2 y^{p_1}}{\sigma^2} \left( - r p_2 - \lambda \right) + \lim_{y\to\infty} L y^{p_1-1} \left(  \left(\frac{2r}{\sigma^2}-1\right) p_2 + \frac{2\lambda}{\sigma^2}  \right). \label{BIU}
\end{align}

With a little more effort, we could determine the sign of $ - r p_2 - \lambda$,
\begin{align*}
 - r p_2 - \lambda &= \frac{- \sqrt{  r^2 \left( r+\frac{\sigma^2}{2}  \right)^2 + (\lambda \sigma^2) ^2 + 2 r^2 \lambda \sigma^2 + r\lambda \sigma^4 }+  \sqrt{ r^2 \left( r+\frac{\sigma^2}{2} \right)^2 + 2 r^2 \lambda \sigma^2} }{\sigma^2} 
 <0,
\end{align*}
from which, together with $p_1>p_1-1$, we know that the first term of \eqref{BIU} heads towards $-\infty$ more rapidly than the second term heading towards $+\infty$ (as $\frac{2r}{\sigma^2} - 1 < 0$ and $p_2<0$) and thus, $g(y)\to -\infty$ as $y\to\infty$. Finally, we appeal to the intermediate value theorem, together with the fact that $y\mapsto g(y)$ is decreasing, establishing the uniqueness of $b \in (\max\{1, L\}, \infty)$.
\end{proof}

We are now ready to perform some numerical experiments and obtain some economic insights. Figures~\ref{F1} and~\ref{F2} show that default risk acts mainly on the level of the option value rather than on the exercise rule. For the baseline parameters, $V$ lies well below $V_{std}$, whereas the threshold only moves from $b_{std}\approx 5.2768$ to $b\approx 5.0845$. In all cases we tested, $b<b_{std}$, that is, default risk leads to earlier exercise. This is consistent with the fact that $y^\alpha$ is decreasing in $y$, the gain function grows more slowly than that of the standard Russian option and the reward for waiting until the asset price falls further below its running maximum is reduced. In Figure \ref{F2}, we see that the difference $b_{std}-b$ is mainly governed by $\alpha = \frac{2r}{\sigma^2}-1$: it vanishes as $\alpha\to 0$, recovering the reward of the standard Russian option (see also Remark \ref{RCOVERRO}); and it decreases with $r$ and increases with $\sigma$, since the survival factor $y^{\alpha}$ decays faster (corresponding to smaller $r$ and larger $\sigma$) in $y$ and the benefit of waiting is partly offset. The difference is also decreasing in $\lambda$ since the benefit of waiting is discounted more heavily, and the additional effect of default risk becomes smaller. Finally, both $b$ and $b_{std}$ increase almost linearly with $L$, since the gain function is zero for $y\leq L$ and the exercise region must lie above $L$; the difference $b_{std} - b$ only increases slightly with $L$, so parameter $L$ mainly shifts both thresholds rather than changing the effect of default risk.

\begin{figure}[H]
\centering
\includegraphics[width=10cm]{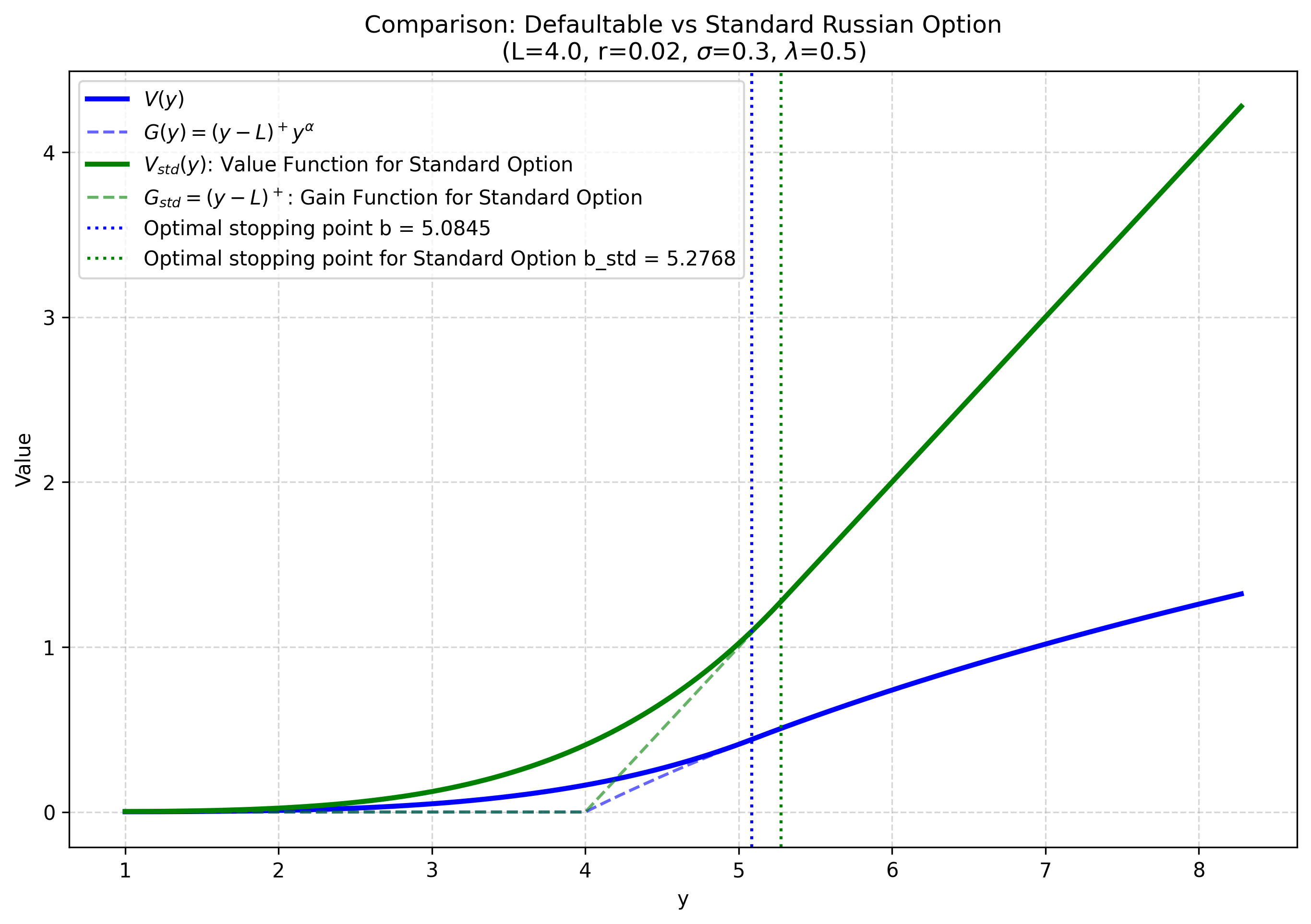}
\caption{This figure displays the maps $y\mapsto V(y)$ and $y\mapsto G(y)$, $y\mapsto V_{std}(y)$ and $y\mapsto G_{std}(y)$ with chosen parameters $L=4$, $r=0.02$, $\sigma=0.3$, $\lambda=0.5$, where $V_{std}$ and $G_{std}$ are the value function and gain function for standard Russian option and are computed via \cite[Page 775]{Shiryaev1999}.} \label{F1}
\end{figure}

\begin{figure}[H]
\centering
\includegraphics[width=13cm]{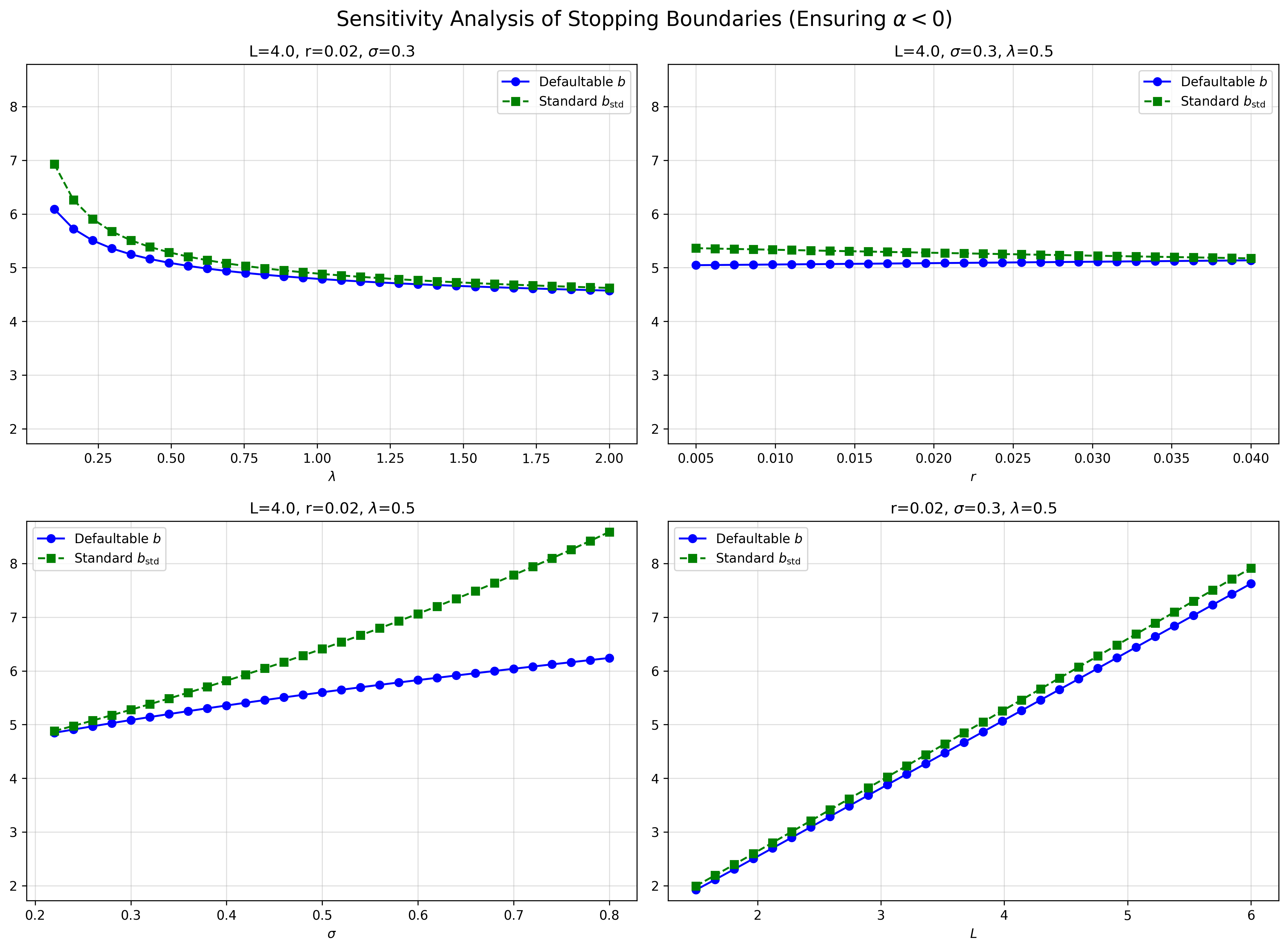}
\caption{This figure displays the sensitivity of the optimal threshold $b$ with respect to $\lambda$, $r$, $\sigma$ and $L$, with the remaining parameters fixed at $L=4$, $r=0.02$, $\sigma=0.3$, $\lambda=0.5$ and always satisfying
$r-\sigma^2/2<0$.} \label{F2}
\end{figure}

To conclude, this paper has studied a defaultable perpetual Russian option in which default occurs at the last time the asset price attains its running maximum. Since this default time is generally not a stopping time, the problem differs from the classical perpetual Russian option and requires a treatment that accounts for the random default horizon. By employing the Az\'{e}ma supermartingale associated with the last exit time and an appropriate change of measure, we have transformed the original problem into an optimal stopping problem for a reflected diffusion. This reformulation allows us to characterise the value function and the optimal stopping boundary. We have also established the uniqueness of the boundary, thereby providing a complete characterisation of the optimal exercise strategy. Our results illustrate how a last-exit-time default mechanism can be incorporated into the valuation of perpetual American-style claims and how the resulting problem can be analysed using optimal stopping techniques for reflected diffusions. A natural direction for future research is to extend the present framework to a finite-horizon setting and investigate how the introduction of a recovery process, that is, a predictable process specifying the payoff received if default occurs before exercise time, affects the resulting optimal stopping problem.

\section{Declarations}

\noindent

\textbf{Funding}  
The authors received no financial support for the research, authorship, and/or publication of this article.

\textbf{Conflicts of interest}  
The authors declare that there is no conflict of interest.

\textbf{Ethics approval}  
Not applicable.

\textbf{Informed consent}  
Not applicable.

\textbf{Data availability}  
Data sharing not applicable to this article as no datasets were generated or analyzed.

\end{document}